\documentclass{amsart}
\usepackage[utf8]{inputenc}
\usepackage{amssymb}
\usepackage[english]{babel}
\usepackage{mathrsfs}
\usepackage{euscript}
\usepackage{orcidlink}
\usepackage{cite}

\usepackage{amsaddr}
\usepackage{amssymb}
\usepackage{amsthm}
\usepackage{amsfonts}
\usepackage{amsbsy}
\usepackage{enumitem}
\usepackage{xcolor}
\renewcommand{\le}{\leqslant}
\renewcommand{\ge}{\geqslant}

\newtheorem{theorem}{Theorem}[section]

\newtheorem{lemma}[theorem]{Lemma}

\newtheorem{corollary}[theorem]{Corollary}
\theoremstyle{definition}
\newtheorem{example}[theorem]{Example}

\newtheorem{remark}[theorem]{Remark}
\newcommand{\tocorrect}[1]{}
\setlist[enumerate,1]{label=(\roman*)}
\numberwithin{equation}{section}

\begin{document}
	\title{Every expansive $ m $-concave operator has $ m $-isometric dilation}
	\author{Michał Buchała\orcidlink{0000-0001-5272-9600}}
	\email{mbuchala@agh.edu.pl}
	
	\address{AGH University of Krakow, Faculty of Applied Mathematics, al. A. Mickiewicza 30, 30-059 Krakow}
	\begin{abstract}
        The aim of this paper is to obtain $ m $-isometric dilation of expansive $ m $-concave operator on Hilbert space. The obtained dilation is shown to be minimal. The matrix representation of this dilation is given. It is also proved that in case of 3-concave operators the assumption on expansivity is not necessary. The paper contains an example showing that minimal $ m $-isometric dilations may not be isomorphic.
	\end{abstract}
   	\maketitle
	\section{Introduction}
\label{SecIntroduction}
$ m $-isometric operators were introduced by Agler in \cite{aglerDisconjugacyTheoremToeplitz}; these are operators satisfying the equality
\begin{equation*}
    \sum_{k=0}^{m} (-1)^{m-k}\binom{m}{k}T^{\ast k}T^{k} = 0
\end{equation*}
and can be viewed as a certain generalization of isometric operators (in particular, 1-isometries are precisely isometries). The class of $ m $-isometries plays a significant role in operator theory and has been studied extensively by many authors (see e.g. \cite{bermudezLocalSpectralPropertiesOfMIsometric,bermudezProductsMisometries,jablonskiMisometricOperatorsTheirSpectralProperties,anandSolutionToCauchyDualSubnormalityProblem,RichterRepresentationCyclic2Isometries}). We refer to \cite{aglerMIsometricTransformationsOfHilbertSpaceI,aglerMIsometricTransformationsOfHilbertSpaceII,aglerMIsometricTransformationsOfHilbertSpaceIII} for the comprehensive study of fundamental properties of such operators. \par
The celebrated Sz.-Nagy-Foias theorem (see \cite[Chapter I, Theorem 4.1]{SzNagyHarmonicAnalysisOfOperators}) states that every contractive operator has an isometric dilation, that is, for every $ T\in \mathbf{B}(H) $ with $ \lVert T\rVert \le 1 $ there exists an isometric operator $ V\in \mathbf{B}(K) $ on a larger Hilbert space $ K\supset H $ satisfying
\begin{equation*}
    P_{H}V^{n}|_{H} = T^{n}, \qquad n\in \mathbb{N};
\end{equation*}
the matrix representation of the dilation takes the form
\begin{equation}
    \label{FormDilationSzNagy}
    V = \begin{bmatrix}
        T & 0 & 0 & \ddots\\
        D_{T} & 0 & 0 & \ddots\\
        0 & I & 0 & \ddots\\
        \ddots & \ddots & \ddots & \ddots
    \end{bmatrix},
\end{equation}
where $ D_{T} = (I-T^{\ast}T)^{\frac{1}{2}} $. Dilations and liftings have been an object of long-term research in operator theory and have found many applications (see e.g. \cite{foiasCommutantLiftingInterpolationProblem}). In \cite[Theorem 5.80]{aglerMIsometricTransformationsOfHilbertSpaceII} Agler and Stankus proved that every 2-isometry possesses so-called Brownian unitary lifting. In \cite{badeaCauchyDual2IsometricLiftingsOfConcave} and \cite{badeaHilbertSpaceOperatorsWith2IsometricDilations} Badea and Suciu studied 2-isometric liftings and showed, in particular, that every 2-concave operator has 2-isometric lifting; the matrix representation of the obtained lifting is quite similar to \eqref{FormDilationSzNagy}. In \cite{suciuCouplingOfOperatorsWith2IsometriesIn3IsometricLiftings} the authors investigated 3-isometric liftings with 2-isometric part; the obtained characterization of existence of such a lifting was applied, in particular, to show that expansive 3-concave operators have 3-isometric liftings. In \cite{badeaHighOrderIsometricLiftings} it was shown that the operators, the powers of which grow polynomially, have $ m $-isometric liftings for some $ m\in \mathbb{N}_{3} $; some special cases were also studied -- in particular, it was proved that each 2-convex operator $ T\in \mathbf{B}(H) $, which satifies the necessary condition
\begin{equation*}
    \sup_{n\in \mathbb{N}_{1}} \frac{\lVert T^{n}\rVert^{2}}{n} < \infty,
\end{equation*}
has a 2-isometric lifting (see \cite[Theorem 3.2]{badeaHighOrderIsometricLiftings}). In \cite{suciuOperatorsExpansiveMisometric} expansive $ m $-isometric liftings with isometric part were investigated. For more results on dilations and liftings we refer also to \cite{suciuBrownianExtensionsThreeIsometries,suciuOnOperatorsWith2IsometricLiftings,suciuTriangulationsOfOperators2IsometricLiftings}.\par
In this paper we are interested in $ m $-isometric dilations of the form
\begin{equation}
    \label{FormDilationMatrixIntr}
    W = \begin{bmatrix}
        T & 0 & 0 & \ddots\\
        U & 0 & 0 & \ddots\\
        0 & S_{1} & 0 & \ddots\\
        0 & 0 & S_{2} & \ddots\\
        \ddots & \ddots & \ddots & \ddots
    \end{bmatrix},
\end{equation}
where $ U $ is nonnegative and $ (S_{j})_{j\in \mathbb{N}_{1}} $ is the sequence of positive and invertible weights of $ m $-isometric unilateral weighted shift. The paper is organized as follows. Section \ref{SecPreliminaries} is devoted to set up a notation and terminology and to review some of standard facts in operator theory. Section \ref{SecMIsomDilations} contains the main results. In Theorem \ref{ThmDilationOfMConcave} we prove that expansive $ m $-concave operators, that is, the operators satisfying the inequalities
\begin{align*}
    &T^{\ast}T - I \ge 0,\\
    &\sum_{k=0}^{m} (-1)^{m-k}\binom{m}{k}T^{\ast k}T^{k} \le 0,
\end{align*}
have $ m $-isometric dilations of the form \eqref{FormDilationMatrixIntr}. In the particular case of 3-concave operators we show that the assumption on expansivity can be dropped (see Theorem \ref{ThmDilationOf3ConcaveWithoutExpansivity}), which improves \cite[Theorem 3.1]{suciuCouplingOfOperatorsWith2IsometriesIn3IsometricLiftings}. We investigate the minimality of the dilation \eqref{FormDilationMatrixIntr} -- Lemma \ref{LemMinimalDilations} implies that the dilations obtained in Theorems \ref{ThmDilationOfMConcave} and \ref{ThmDilationOf3ConcaveWithoutExpansivity} are minimal. At the end, we give an example of two 2-isometric dilations of 2-concave operator, which are minimal but not isomorphic.
	\section{Preliminaries}
\label{SecPreliminaries}
In what follows, $ \mathbb{Z} $ and $ \mathbb{N} $ stand for the sets of integers and non-negative integers, respectively. For $ k\in \mathbb{N} $ we set
\begin{equation*}
    \mathbb{N}_{k} = \{n\in \mathbb{N}\colon n\ge k \}.
\end{equation*}
By $ \mathbb{R} $ and $ \mathbb{C} $ we denote the fields of real and complex numbers respectively. For simplicity we write $ (a_{i})_{i\in I} \subset A $, if $ (a_{i})_{i\in I} $ is the sequence of elements of a set $ A $.\par
If $ R $ is an algebra, then $ R[z] $ stands for the algebra of all polynomials with coefficients in $ R $ with standard addition and multiplication. For $ n\in \mathbb{N} $, $ R_{n}[z] $ denotes the vector space of all polynomials with coefficients in $ R $ of degree less than or equal to $ n $.\par
Throughout this paper all Hilbert spaces are assumed to be complex. Let $ H $ be a Hilbert space. If $ (M_{n})_{n\in \mathbb{N}} $ is a sequence of subspaces of $ H $, then $ \bigvee_{n=0}^{\infty}M_{n} $ denotes the closed subspace spanned by $ \bigcup_{n=0}^{\infty}M_{n} $. By $ \mathbf{B}(H) $ we denote the $ C^{\ast} $-algebra of all linear and bounded operators on $ H $. If $ T\in \mathbf{B}(H) $, then $ \mathcal{N}(T) $ and $ \mathcal{R}(T) $ stands for the kernel and the range of $ T $, respectively; we have the following so-called range-kernel decomposition:
\begin{equation}
    \label{FormRangeKernelDecomp}
    H = \mathcal{N}(T)\oplus \overline{\mathcal{R}(T^{\ast})}.
\end{equation}
We say that $ T\in \mathbf{B}(T) $ is non-negative (and we write $ T\ge 0 $) if $ \langle Tf,f\rangle \ge 0 $ for all $ f\in H $; we call $ T $ positive if $ \langle Tf,f\rangle > 0 $ for all $ f\in H $, $ f\not=0 $. For $ T\in \mathbf{B}(H) $ and $ m\in \mathbb{N}_{1} $ we set
\begin{equation*}
    \beta_{m}(T) := \sum_{k=0}^{m} (-1)^{m-k}\binom{m}{k}T^{\ast k}T^{k}.
\end{equation*}
Recall that an operator $ T\in \mathbf{B}(H) $ is $ m $-concave ($ m\in \mathbb{N}_{1} $) if $ \beta_{m}(T) \le 0 $ and $ T $ is $ m $-isometric if $ \beta_{m}(T) = 0 $. We say that $ T\in \mathbf{B}(H) $ is expansive if $ \lVert Tf\rVert \ge \lVert f\rVert $ for all $ f\in H $ (equivalently, $ \beta_{1}(T) \ge 0 $).\par
If $ K $ is a  Hilbert spaces such that $ H\subset K $ (in sense of isometric embeddings) and $ T\in \mathbf{B}(H) $, $ S\in \mathbf{B}(K) $, then $ S $ is called a dilation of $ T $ if $ P_{H}S^{n}|_{H} = T^{n} $ for every $ n\in \mathbb{N} $, where $ P_{H} $ is the orthogonal projection of $ K $ onto $ H $. The dilation $ S\in \mathbf{B}(K) $ of an operator $ T\in \mathbf{B}(H) $ is called minimal if 
\begin{equation*}
    K = \bigvee_{n=0}^{\infty} S^{n}H.
\end{equation*}
Suppose $ S_{1}\in \mathbf{B}(K) $ and $ S_{2}\in \mathbf{B}(K^{\prime}) $ are dilations of $ T\in \mathbf{B}(H) $. Following \cite[Chapter I]{SzNagyHarmonicAnalysisOfOperators} we call the dilations $ S_{1} $ and $ S_{2} $ isomorphic if there exists a unitary operator $ V\colon K\to K^{\prime} $ such that
\begin{align*}
    Vh &= h, \qquad h\in H,\\
    VS_{1} &= S_{2}V. 
\end{align*}\par
Define
\begin{equation*}
    \ell^{2}(\mathbb{N},H) := \{(f_{j})_{j\in \mathbb{N}}\subset H\colon \sum_{j=0}^{\infty} \lVert f_{j}\rVert^{2} < \infty \};
\end{equation*}
this is a Hilbert space with the inner product given by
\begin{equation*}
    \langle f,g\rangle := \sum_{j=0}^{\infty} \langle f_{j},g_{j}\rangle, \quad f = (f_{j})_{j\in \mathbb{N}},g = (g_{j})_{j\in \mathbb{N}} \in \ell^{2}(\mathbb{N},H).
\end{equation*}
Every operator $ T\in \mathbf{B}(\ell^{2}(\mathbb{N},H)) $ has a matrix representation $ [T_{i,j}]_{i,j\in \mathbb{N}} $, where $ T_{i,j}\in \mathbf{B}(H) $, satisfying the following formula (see \cite[Chapter 8]{halmosHilbertSpaceProblemBook}):
\begin{equation*}
    Tf = \left( \sum_{j=0}^{\infty} T_{i,j} f_{j} \right)_{i\in \mathbb{N}}, \quad f = (f_{j})_{j\in \mathbb{N}} \in \ell^{2}(\mathbb{N},H).
\end{equation*}
If $ (S_{j})_{j\in \mathbb{N}_{1}}\subset \mathbf{B}(H) $ is a uniformly bounded sequence of operators, then we define the unilateral weighted shift $ S\in \mathbf{B}(\ell^{2}(\mathbb{N},H)) $ with weights $ (S_{j})_{j\in \mathbb{N}_{1}} $ as follows
\begin{equation*}
    Sf = (0,S_{1}f_{0},S_{2}f_{1},\ldots), \quad f=(f_{j})_{j\in \mathbb{N}} \in \ell^{2}(\mathbb{N},H).
\end{equation*}
The matrix representation of $ S $ takes the form
\begin{equation*}
    S = \begin{bmatrix}
        0 & 0 & 0 & 0 & \ddots\\
        S_{1} & 0 & 0 & 0 & \ddots \\
        0 & S_{2} & 0 & 0 & \ddots\\
        0 & 0 & S_{3} & 0 & \ddots\\
        \ddots & \ddots & \ddots & \ddots & \ddots
    \end{bmatrix}.
\end{equation*}

	\section{$ m $-isometric dilations of $ m $-concave operators}
\label{SecMIsomDilations}
The main result of this paper is stated below. It is a generalization of \cite[Theorem 3.1]{suciuCouplingOfOperatorsWith2IsometriesIn3IsometricLiftings}.
\begin{theorem}
    \label{ThmDilationOfMConcave}
    Let $ H $ be a Hilbert space and let $ T\in \mathbf{B}(H) $ be expansive and $ m $-concave ($ m\in \mathbb{N}_{2} $). Then $ T $ has an $ m $-isometric dilation $ W $ on the space $ H\oplus \ell^{2}(\mathbb{N}, H^{\prime}) $, where $ H' \subset H $ is a closed subspace of $ H $, with the following matrix representation
    \begin{equation}
        \label{FormDilationMatrix}
        W = \begin{bmatrix}
            T & 0 & 0 & 0 &  \ddots\\
            U & 0 & 0 & 0 &  \ddots\\
            0 & S_{1} & 0 & 0 &  \ddots \\
            0 & 0 & S_{2} & 0 &  \ddots\\
            0 & 0 & 0 & S_{3} &  \ddots\\
            \ddots & \ddots & \ddots & \ddots & \ddots
        \end{bmatrix},
    \end{equation}
    where $ U\in \mathbf{B}(H,H^{\prime}) $ is nonnegative and $ S_{j}\in \mathbf{B}(H^{\prime}) $ is positive and invertible for every $ j\in \mathbb{N}_{1} $. Moreover, $ U $ can be chosen to satisfy the equality $ T^{\ast}U^{2}T = U^{2} $ and $ S_{1},\ldots,S_{m-2} $ can be chosen to be identity operators on $ H^{\prime} $.
\end{theorem}
Before we prove the above theorem we need several technical lemmata. In the first lemma we give the formula for powers of operator $ W $ given by the matrix \eqref{FormDilationMatrix}; the inductive proof of this formula is left to the reader.
\begin{lemma}
    \label{LemMatrixRepresentationOfPowers}
    Let $ W\in \mathbf{B}(H\oplus\ell^{2}(\mathbb{N},H^{\prime})) $ be as in \eqref{FormDilationMatrix}. Then for $ m\in \mathbb{N}_{1} $ and $ h = (h_{0},h_{1},h_{2},\ldots)\in H\oplus \ell^{2}(\mathbb{Z},H^{\prime}) $, the $ k^{th} $ entry of $ W^{m}h $ is of the form
    \begin{equation}
        \label{FormDilationMatrixPowers}
        (W^{m}h)_{k} = \begin{cases}
            T^{m}h_{0}, & k = 0\\
            UT^{m-1}h_{0}, & k = 1\\
            S_{k-1}\cdots S_{1}UT^{m-k}h_{0}, & 2\le k \le m\\
            S_{k-1}\cdots S_{k-m}h_{k-m}, & k > m
        \end{cases}.
    \end{equation}
\end{lemma}
From the above result we can easily derive the following simple condition for $ W $ to be $ m $-isometric.
\begin{lemma}
    \label{LemMIsometricityOfDilation}
    Let $ W\in \mathbf{B}(H\oplus\ell^{2}(\mathbb{N},H^{\prime})) $ be as in \eqref{FormDilationMatrix}. For $ m \in \mathbb{N}_{1} $ the following conditions are equivalent:
    \begin{enumerate}
        \item $ W $ is $ m $-isometric,
        \item the unilateral weighted shifts $ S\in \mathbf{B}(\ell^{2}(\mathbb{N},H^{\prime})) $ with weights $ (S_{j})_{j\in \mathbb{N}_{1}} $ is $ m $-isometric and
        \begin{equation}
            \label{FormDilationMatrixMIsometric}
            \langle \beta_{m}(T)h,h\rangle +\sum_{\ell=1}^{m}(-1)^{m-\ell}\binom{m}{\ell}\sum_{k=1}^{\ell} \lVert S_{k-1}\cdots S_{1}UT^{\ell-k}h\rVert^{2} = 0.
        \end{equation}
    \end{enumerate}
\end{lemma}
If $ Q\in \mathbf{B}(H) $ is nonnegative, then we say that the operator $ T\in \mathbf{B}(H) $ is $ Q $-isometric if $ T^{\ast}QT = Q $. The next lemma states that expansive $ m $-concave operators are actually $ Q $-isometric.
\begin{lemma}
    \label{LemMConcaveAreQIsometries}
    Let $ H $ be a Hilbert space and let $ T\in \mathbf{B}(H) $ be expansive and $ m $-concave. Then there exists a nonnegative operator $ Q\in \mathbf{B}(H) $ such that $ Q\ge \beta_{m-1}(T) $ and $ T^{\ast} QT = Q $.
\end{lemma}
\begin{proof}
    Set $ \Delta := \beta_{m-1}(T) $. We infer from \cite[Theorem 2.5]{guOnMExpansiveMContractiveOperators} that $ \Delta \ge 0 $. Then $ m $-concavity can be written equivalently as $ T^{\ast}\Delta T \le \Delta $, which implies that
    \begin{equation}
        \label{ProofFormMConcaveContractive}
        \lVert \Delta^{\frac{1}{2}}Th\rVert \le \lVert \Delta^{\frac{1}{2}}h\rVert, \quad h\in H. 
    \end{equation}
    Define a sesquilinear form $ \varphi\colon \mathcal{R}(\Delta^{\frac{1}{2}})\times \mathcal{R}(\Delta^{\frac{1}{2}}) \to \mathbb{C} $ as follows:
    \begin{equation*}
        \varphi\left(\Delta^{\frac{1}{2}}f,\Delta^{\frac{1}{2}}g\right) = \left\langle \Delta^{\frac{1}{2}}Tf, \Delta^{\frac{1}{2}}g\right\rangle, \quad f,g\in H.
    \end{equation*}
    It follows from \eqref{ProofFormMConcaveContractive} that $ \mathcal{N}(\Delta^{\frac{1}{2}}) $ is $ T $-invariant. Hence, $ \varphi $ is well-defined. Moreover,
    \begin{align*}
        \left\lvert \varphi\left( \Delta^{\frac{1}{2}}f,\Delta^{\frac{1}{2}}g \right)\right\rvert &\le \left \lVert \Delta^{\frac{1}{2}}Tf\right \rVert \cdot \left \lVert \Delta^{\frac{1}{2}} g\right \rVert \\
        &\stackrel{\eqref{ProofFormMConcaveContractive}}{\le} \left \lVert \Delta^{\frac{1}{2}}f\right \rVert \cdot \left \lVert \Delta^{\frac{1}{2}} g\right \rVert.
    \end{align*}
    Thus, $ \varphi $ is continuous, so it can be extended to a continuous sesquilinear form on $ \overline{\mathcal{R}(\Delta^{\frac{1}{2}})} $, which, for the convenience will be denoted also by $ \varphi $. By the Riesz theorem \cite[Theorem 12.8]{rudinFunctionalAnalysis}, there exists an operator $ T_{1}\in \mathbf{B}(\overline{\mathcal{R}(\Delta^{\frac{1}{2}})}) $ such that
    \begin{equation*}
        \varphi\left( \Delta^{\frac{1}{2}}f,\Delta^{\frac{1}{2}}g \right) = \langle T_{1}\Delta^{\frac{1}{2}}f,\Delta^{\frac{1}{2}}g\rangle, \quad f,g\in H.
    \end{equation*}
    From the above and the definition of $ \varphi $ it follows that
    \begin{equation}
        \label{ProofFormCommutation}
        \Delta T - \Delta^{\frac{1}{2}}T_{1}\Delta^{\frac{1}{2}} = 0.
    \end{equation}
    In turn, for $ f\in H $ we have
    \begin{equation*}
        \Delta Tf - \Delta^{\frac{1}{2}}T_{1}\Delta^{\frac{1}{2}}f = \Delta^{\frac{1}{2}}(\Delta^{\frac{1}{2}}Tf - T_{1}\Delta^{\frac{1}{2}}f).
    \end{equation*}
    Since $ \Delta^{\frac{1}{2}}Tf - T_{1}\Delta^{\frac{1}{2}}f \in \overline{\mathcal{R}(\Delta^{\frac{1}{2}})} $ and $ \overline{\mathcal{R}(\Delta^{\frac{1}{2}})}\perp \mathcal{N}(\Delta^{\frac{1}{2}}) $, we deduce from \eqref{ProofFormCommutation} that
    \begin{equation}
        \Delta^{\frac{1}{2}}T = T_{1}\Delta^{\frac{1}{2}}.
    \end{equation}
    By \eqref{ProofFormMConcaveContractive}, it follows from the above that
    \begin{equation*}
        \lVert T_{1}\Delta^{\frac{1}{2}}f \rVert \le \lVert \Delta^{\frac{1}{2}}f\rVert, \quad f\in H,
    \end{equation*}
    so $ T_{1} $ is a contraction. Using the fact that $ T $ is expansive we infer from \cite[Theorem 2.1]{treilFixedPointNehariProblem} and \cite[Theorem 2.1]{biswasWeightedCommutantLifting} that there exists a nonnegative operator $ Q\in \mathbf{B}(H) $ such that $ T^{\ast}QT = Q \ge \Delta $, which gives us the claim.
\end{proof}
Now we are in the position to prove Theorem \ref{ThmDilationOfMConcave}.
\begin{proof}[Proof of Theorem \ref{ThmDilationOfMConcave}]
    Let $ Q\in \mathbf{B}(H) $ be the operator given by Lemma \ref{LemMConcaveAreQIsometries}. Define a sesquilinear form $ \psi\colon \mathcal{R}(Q^{\frac{1}{2}})\times \mathcal{R}(Q^{\frac{1}{2}}) \to \mathbb{C} $ as follows:
    \begin{equation*}
        \psi(Q^{\frac{1}{2}}f,Q^{\frac{1}{2}}g) = \langle \beta_{m}(T)f,g\rangle, \quad f,g\in H.
    \end{equation*}
    Set $ \Delta := \beta_{m-1}(T) $. Again, we have $ \Delta \ge 0 $. It is a matter of routine to verify that $ \beta_{m}(T) = T^{\ast} \Delta T - \Delta $. Hence,
    \begin{equation*}
        \psi(Q^{\frac{1}{2}}f,Q^{\frac{1}{2}}g) = \langle \Delta^{\frac{1}{2}}Tf,\Delta^{\frac{1}{2}}Tg\rangle - \langle \Delta^{\frac{1}{2}}f,\Delta^{\frac{1}{2}}g\rangle.
    \end{equation*}
    Since $ T^{\ast}QT = Q \ge \Delta \ge T^{\ast}\Delta T $, it follows that
    \begin{align*}
        \lvert \psi(Q^{\frac{1}{2}}f,Q^{\frac{1}{2}}g) \rvert &\le \lvert  \langle \Delta^{\frac{1}{2}}Tf,\Delta^{\frac{1}{2}}Tg\rangle\rvert + \lvert \langle \Delta^{\frac{1}{2}}f,\Delta^{\frac{1}{2}}g\rangle\rvert\\
        &\le \lVert \Delta^{\frac{1}{2}}Tf\rVert \lVert \Delta^{\frac{1}{2}}Tg\rVert + \lVert \Delta^{\frac{1}{2}}f\rVert \lVert \Delta^{\frac{1}{2}}g \rVert\\
        & \le 2\lVert Q^{\frac{1}{2}}f\rVert\lVert Q^{\frac{1}{2}}g\rVert, \quad f,g\in H.
    \end{align*}
    Therefore, $ \psi $ is a well-defined continuous sesquilinear form on $ \mathcal{R}(Q^{\frac{1}{2}}) $, so it can be extended to a continuous sesquilinear form on $ \overline{\mathcal{R}(Q^{\frac{1}{2}})} $, which will be denoted by the same letter $ \psi $. By the Riesz theorem, there exists an operator $ A\in \mathbf{B}(\overline{\mathcal{R}(Q^{\frac{1}{2}})}) $ such that 
    \begin{equation*}
        \psi(Q^{\frac{1}{2}}f,Q^{\frac{1}{2}}g) = \langle AQ^{\frac{1}{2}}f,Q^{\frac{1}{2}}g\rangle, \quad f,g\in H.
    \end{equation*}
    Since $ T $ is $ m $-concave, we have $ -A \ge 0 $, which implies that $ I-A\ge I $. Set
    \begin{equation*}
        B := (I-A)^{\frac{1}{2}}\in \mathbf{B}(\overline{\mathcal{R}(Q^{\frac{1}{2}})}).
    \end{equation*}
    Then $ B $ is positive and expansive, which implies that $ B $ is invertible. Define the polynomial $ p\in \mathbf{B}(\overline{\mathcal{R}(Q^{\frac{1}{2}})})_{m-1}[z] $ by the formula:
    \begin{equation*}
        p(z) = \frac{1}{(m-1)!}z(z-1)\cdots(z-m+2)(-A)+I, \quad z\in \mathbb{C}.
    \end{equation*}
    Then $ p(k) = I $ for $ k\in \mathbb{N}\cap[0,m-2] $, $ p(m-1) = B^{2} $ and
    \begin{equation*}
        p(n) = \frac{n(n-1)\cdots (n-m+2)}{(m-1)!}(-A)+I\ge I, \quad n\in \mathbb{N}_{m}.
    \end{equation*}
    Thus, $ p(n) $ is positive and invertible for every $ n\in \mathbb{N} $. Next, for $ h\in \overline{\mathcal{R}(Q^{\frac{1}{2}})} $, $ h\not=0 $, we have $ \langle p(n+1)h,h\rangle = \langle p(n)h,h\rangle = \langle h,h\rangle $ for $ n\in \mathbb{N}\cap[0,m-3] $ and
    \begin{equation*}
        \frac{\langle p(m-1)h,h\rangle}{\langle p(m-2)h,h\rangle} = \frac{\lVert Bh\rVert^{2}}{\lVert h\rVert^{2}} \le \lVert B \rVert^{2}.
    \end{equation*}
    In turn, if we set\footnote{Note that both numerator and denominator are positive and are the polynomials in $ n $ of the same degree, so the supremum is indeed finite.}
    \begin{equation*}
        C:= \sup_{n\in \mathbb{N}_{m-1}} \frac{(n+1)n\cdots(n-m+3)}{n(n-1)\cdots(n-m+2)} \in [1,\infty),
    \end{equation*}
    then
    \begin{align*}
        \frac{\langle p(n+1)h,h\rangle}{\langle p(n)h,h\rangle} &= \frac{(n+1)\cdots(n-m+3)\langle -Ah,h\rangle + (m-1)!\lVert h\rVert^{2}}{n\cdots(n-m+2)\langle -Ah,h\rangle + (m-1)!\lVert h\rVert^{2}}\\
        &\le \frac{Cn\cdots(n-m+1)\langle -Ah,h\rangle + (m-1)!\lVert h\rVert^{2}}{n\cdots(n-m+1)\langle -Ah,h\rangle + (m-1)!\lVert h\rVert^{2}} \le C
    \end{align*}
    for every $ n\in \mathbb{N}_{m-1} $ and $ h\in \overline{\mathcal{R}(Q^{\frac{1}{2}})} $, $ h\not=0 $. Therefore,
    \begin{equation*}
        \sup\left\{ \frac{\langle p(n+1)h,h\rangle}{\langle p(n)h,h\rangle}\colon n\in \mathbb{N}, h\in \overline{\mathcal{R}(Q^{\frac{1}{2}})}, h\not=0 \right\} \le \max\left\{ \lVert B\rVert^{2},C \right\}.
    \end{equation*} By \cite[Theorem 3.4]{buchalaMIsometricShiftsOperatorWeights} there exists an $ m $-isometric unilateral weighted shift with positive and invertible weights $ (S_{j})_{j\in \mathbb{N}_{1}}\subset \mathbf{B}(\overline{\mathcal{R}(Q^{\frac{1}{2}})}) $ such that
    \begin{equation}
        \label{ProofFormWeightOfUWSPolynomial}
        p(n) = \lvert S_{n}\cdots S_{1}\rvert^{2},\quad n\in \mathbb{N}_{1}.
    \end{equation}
    Since $ p(n)= I $ for $ n\in \mathbb{N}\cap[0,m-2] $, $ S_{1}= \ldots = S_{m-2} = I $. Combining this with the equality $ p(m-1) = B^{2} $ we obtain that $ B = S_{m-1} $. Set $ H^{\prime} = \overline{\mathcal{R}(Q^{\frac{1}{2}})} $ and let $ W\in \mathbf{B}(H\oplus \ell^{2}(\mathbb{N},H^{\prime})) $ be as in \eqref{FormDilationMatrix} with $ U = Q^{\frac{1}{2}} $ and $ (S_{j})_{j\in \mathbb{N}_{1}} $ constructed above. In view of Lemma \ref{LemMIsometricityOfDilation}, in order to prove that $ W $ is $ m $-isometric it suffices to prove that \eqref{FormDilationMatrixMIsometric} holds. Since $ T^{\ast}QT = Q $ and $ S_{m-1} = B $, it follows that \eqref{FormDilationMatrixMIsometric} takes the form
    \begin{equation}
        \label{ProofFormDilationMatrixMIsometricEquiv}
        \langle \beta_{m}(T)h,h\rangle+\sum_{\ell=1}^{m}(-1)^{m-\ell}\binom{m}{\ell}\sum_{k=1}^{\ell} \lVert Q^{\frac{1}{2}}h\rVert^{2} -\lVert Q^{\frac{1}{2}}h\rVert^{2}+\lVert BQ^{\frac{1}{2}}h\rVert^{2} = 0.
    \end{equation}
    Note that
    \begin{align*}
        \lVert BQ^{\frac{1}{2}}h\rVert^{2} - \lVert Q^{\frac{1}{2}}h\rVert^{2} &= \langle (B^{2}-I)Q^{\frac{1}{2}}h,Q^{\frac{1}{2}}\rangle\\
        &=\langle -AQ^{\frac{1}{2}}h,Q^{\frac{1}{2}}h\rangle\\
        &=-\psi(Q^{\frac{1}{2}}h,Q^{\frac{1}{2}}h)\\
        &=-\langle \beta_{m}(T)h,h\rangle, \quad h\in H.
    \end{align*}
    In turn, standard computation gives us that
    \begin{equation*}
        \sum_{\ell=1}^{m}(-1)^{m-\ell}\binom{m}{\ell}\sum_{k=1}^{\ell} \lVert Q^{\frac{1}{2}}h\rVert^{2} = \lVert Q^{\frac{1}{2}}h\rVert^{2}\sum_{\ell=1}^{m}(-1)^{m-\ell}\binom{m}{\ell}\ell = 0.
    \end{equation*}
    Therefore, \eqref{ProofFormDilationMatrixMIsometricEquiv} holds and, by Lemma \ref{LemMIsometricityOfDilation}, $ W $ is $ m $-isometric. We infer from Lemma \ref{LemMatrixRepresentationOfPowers} that
    \begin{equation*}
        P W^{n}(h,0,\ldots) = T^{n}h, \quad h\in H, \ n\in \mathbb{N},
    \end{equation*}
    where $ P $ is the projection of $ H\oplus \ell^{2}(\mathbb{N},\overline{\mathcal{R}(Q^{\frac{1}{2}})}) $ on the first coordinate. Hence, $ W $ is an $ m $-isometric dilation of $ T $.
\end{proof}
\begin{remark}
    It is worth to note that the operator $ S_{m-1} $ in the dilation $ W $ constructed in the proof of Theorem \ref{ThmDilationOfMConcave} is equal to the identity operator if and only if $ T $ is $ m $-isometric, so that in case of $ T $ being $ m $-concave, but not $ m $-isometric, the unilateral weighted shift with weights $ (S_{j})_{j\in \mathbb{N}_{1}} $ is strictly $ m $-isometric. Indeed, the careful look on the above proof reveals that $ S_{m-1} = I $ if and only if the operator $ A\in \mathbf{B}(\overline{\mathcal{R}(Q^{\frac{1}{2}})}) $ satisfying
    \begin{equation*}
        \langle AQ^{\frac{1}{2}}f,Q^{\frac{1}{2}}g\rangle = \langle \beta_{m}(T)f,g\rangle, \qquad f,g\in H,
    \end{equation*}
    is equal to 0, which is equivalent to $ T $ being $ m $-isometric. Hence, if $ S_{m-1}\not= I $, then the polynomial $ p\in \mathbf{B}(\overline{\mathcal{R}(Q^{\frac{1}{2}})})_{m-1}[z] $ is of degree precisely $ m-1 $.
\end{remark}
It turns out that in case of $ 3 $-concave operators the assumption on expansivity can be dropped and we still can obtain a $ 3 $-isometric dilation.
\begin{theorem}
    \label{ThmDilationOf3ConcaveWithoutExpansivity}
    Let $ H $ be a Hilbert space and let $ T\in \mathbf{B}(H) $ be $ 3 $-concave. Then $ T $ has a $ 3 $-isometric dilation $ W $ of the form \eqref{FormDilationMatrix} on the space $ H\oplus \ell^{2}(H^{\prime}) $, where $ H' \subset H $ is a closed subspace of $ H $.
\end{theorem}
\begin{proof}
    We provide only a sketch of the proof, because the construction of a dilation is very similar to the one given in the proof of Theorem \ref{ThmDilationOfMConcave}. Set $ \Delta = \beta_{2}(T) $. Define a sesquilinear form $ \psi\colon \mathcal{R}(\Delta^{\frac{1}{2}})\times \mathcal{R}(\Delta^{\frac{1}{2}}) \to \mathbb{C} $ by the formula:
    \begin{equation*}
        \psi(\Delta^{\frac{1}{2}}f,\Delta^{\frac{1}{2}}):= \langle T^{\ast}\beta_{3}(T)Tf,g\rangle, \quad f,g\in H.
    \end{equation*}
    Since
    \begin{equation*}
        T^{\ast}\beta_{3}(T)T = T^{\ast 2}\Delta T^{2} - T^{\ast}\Delta T
    \end{equation*}
    and, by 3-concavity,
    \begin{equation*}
        T^{\ast 2}\Delta T^{2} \le T^{\ast }\Delta T \le \Delta
    \end{equation*}
    we get that $ \psi $ is well-defined and continuous.  Extending $ \psi $ to a continuous sesquilinear form on $ \overline{\mathcal{R}(\Delta^{\frac{1}{2}})} $ and applying the Riesz theorem we obtain an operator $ A\in \mathbf{B}(\overline{\mathcal{R}(\Delta^{\frac{1}{2}})}) $ such that
    \begin{equation}
        \langle A\Delta^{\frac{1}{2}}f,\Delta^{\frac{1}{2}}g\rangle = \langle T^{\ast}\beta_{3}(T)Tf,g\rangle, \quad f,g\in H.
    \end{equation}
    It follows from the inequality $ \beta_{3}(T)\le 0 $ that $ A\le 0 $. Setting $ B = (I-A)^{\frac{1}{2}} $ and proceeding as in the proof of Theorem \ref{ThmDilationOfMConcave} we obtain a 3-isometric unilateral weighted shift with positive and invertible weights $ (S_{j})_{j\in \mathbb{N}_{1}} $ such that $ S_{1} = I $ and $ S_{2} = B $. Let $ W $ be as in \eqref{FormDilationMatrix} with $ U = \Delta^{\frac{1}{2}} $ and $ (S_{j})_{j\in \mathbb{N}_{1}} $ given above. By Lemma \ref{LemMIsometricityOfDilation}, it is enough to verify that \eqref{FormDilationMatrixMIsometric} holds to show that $ W $ is 3-isometric. For $ h\in H $ we have
    \begin{align*}
        &\langle \beta_{3}(T)h,h\rangle + \langle T^{\ast 2}\Delta T^{2}h,h\rangle -2\langle T^{\ast}\Delta Th,h\rangle + \langle S_{2}^{2}\Delta^{\frac{1}{2}}h,\Delta^{\frac{1}{2}}h\rangle\\
        &=\langle T^{\ast}\Delta Th,h\rangle - \langle \Delta h,h\rangle + \langle T^{\ast 2}\Delta T^{2}h,h\rangle\\ &-2\langle T^{\ast}\Delta Th,h\rangle + \langle S_{2}^{2}\Delta^{\frac{1}{2}}h,\Delta^{\frac{1}{2}}h\rangle\\
        &=\langle T^{\ast 2}\Delta T^{2}h,h\rangle - \langle T^{\ast}\Delta Th,h\rangle -\langle A\Delta^{\frac{1}{2}}h,\Delta^{\frac{1}{2}}h\rangle\\
        &=\langle T^{\ast}\beta_{3}(T)Th,h\rangle - \psi(\Delta^{\frac{1}{2}}h,\Delta^{\frac{1}{2}}h) = 0.
    \end{align*}
    Hence, \eqref{FormDilationMatrixMIsometric} holds. Therefore, $ W $ is a 3-isometric dilation of $ T $.
\end{proof}
It is well known that an isometric dilation of a contraction can always be chosen to be minimal (see \cite[Chapter I, Theorem 4.1]{SzNagyHarmonicAnalysisOfOperators}). In the next lemma we solve the problem of minimality of the dilations constructed in Theorems \ref{ThmDilationOfMConcave} and \ref{ThmDilationOf3ConcaveWithoutExpansivity}.
\begin{lemma}
    \label{LemMinimalDilations}
    Let $ H $ be a Hilbert space and let $ T\in \mathbf{B}(H) $. Suppose $ U\in \mathbf{B}(H) $ is non-negative and $ (S_{j})_{j\in \mathbb{N}_{1}}\subset \mathbf{B}(\overline{\mathcal{R}(U)}) $ is a uniformly bounded sequence of positive and invertible operators. Let $ W\in \mathbf{B}\left( H\oplus \ell^{2}(\mathbb{N},\overline{\mathcal{R}(U)})  \right) $ be as in \eqref{FormDilationMatrix}. Then
    \begin{equation}
        \label{FormMinimalityCondition}
        H\oplus \ell^{2}(\mathbb{N},\overline{\mathcal{R}(U)}) = \bigvee_{n=0}^{\infty} W^{n}H.
    \end{equation}
\end{lemma}
\begin{proof}
    Let $ K $ be the right hand side of \eqref{FormMinimalityCondition}. It suffices to show that $ K^{\perp} = \{0\} $. Let $ \mathbf{h} = (h_{j})_{j\in \mathbb{N}}\in K^{\perp} $. We will prove by induction that $ h_{j} = 0 $ for every $ j\in \mathbb{N} $. First, note that $ h_{0} = 0 $, because $ (h,0,\ldots) \in K $ for every $ h\in H $. For the proof of the equality $ h_{1} = 0 $ observe that
    \begin{equation*}
        0 = \langle \mathbf{h}, W(h,0,0,\ldots)\rangle = \langle h_{1},Uh\rangle, \qquad h\in H,
    \end{equation*}
    so $ h_{1}\perp \mathcal{R}(U) $. Since we also have $ h_{1}\in \overline{\mathcal{R}(U)} $, it follows that $ h_{1} = 0 $. Now assume that $ h_{j} = 0 $ for $ j\in \mathbb{N}\cap[0,n] $ with some $ n\in \mathbb{N}_{1} $. By Lemma \ref{LemMatrixRepresentationOfPowers}, it follows from the inductive hypothesis that for every $ h\in H $,
    \begin{equation*}
        0 = \langle \mathbf{h}, W^{n+1}(h,0,0,\ldots)\rangle = \langle h_{n+1},S_{n}\cdots S_{1}Uh\rangle = \langle S_{1}\cdots S_{n}h_{n+1},Uh\rangle.
    \end{equation*}
    Hence, $ S_{1}\cdots S_{n}h_{n+1}\perp \mathcal{R}(U) $. On the other hand, $ S_{1}\cdots S_{n}h_{n+1}\in \mathcal{R}(S_{1}) = \overline{\mathcal{R}(U)} $. Therefore, $ S_{1}\cdots S_{n}h_{n+1} = 0 $. Since $ S_{1},\ldots,S_{n} $ are invertible, we have $ h_{n+1} = 0 $, what completes the proof.
\end{proof}
\begin{corollary}
    \label{CorMinimalDilations}
    If $ T\in \mathbf{B}(H) $ is an expansive $ m $-concave operator on a Hilbert space $ H $ ($ m\in \mathbb{N}_{2} $), then $ T $ has a minimal $ m $-isometric dilation of the form \eqref{FormDilationMatrix}. Moreover, if $ T $ is 3-concave (not necessarily expansive), then $ T $ has a minimal 3-isometric dilation of the form \eqref{FormDilationMatrix}.
\end{corollary}
\begin{proof}
    Consult the proof of Theorems \ref{ThmDilationOfMConcave} and \ref{ThmDilationOf3ConcaveWithoutExpansivity} and use Lemma \ref{LemMinimalDilations}.
\end{proof}
Beside the fact that every contraction possess a minimal isometric dilation, it was proved in \cite[Chapter I, Theorem 4.1]{SzNagyHarmonicAnalysisOfOperators} that all minimal isometric dilations are isomorphic. In the following example we will show that it may not hold in the case of 2-isometric dilations.
\begin{example}
    In \cite[Theorem 2.1]{badeaCauchyDual2IsometricLiftingsOfConcave} the author has constructed a 2-isometric dilation of 2-concave operator and this dilation has a matrix representation of the form
    \begin{equation*}
        W^{\prime} = \begin{bmatrix}
            T & 0 & 0 & \ddots\\
            U & 0 & 0 & \ddots\\
            0 & I & 0 & \ddots\\
            0 & 0 & I & \ddots\\
            \ddots & \ddots & \ddots & \ddots 
        \end{bmatrix} \in \mathbf{B}\left( H\oplus \ell^{2}(\mathbb{N},\overline{\mathcal{R}(U)}) \right),
    \end{equation*}
    with $ U_{1} = (Q-\beta_{1}(T))^{\frac{1}{2}} $, where $ Q\in \mathbf{B}(H) $ is as in Lemma \ref{LemMConcaveAreQIsometries}; $ W^{\prime} $ was shown to be a minimal 2-isometric dilation (this follows also immediately from Lemma \ref{LemMinimalDilations}). On the other hand, the dilation constructed in Theorem \ref{ThmDilationOfMConcave} is of the form
    \begin{equation*}
        W^{\prime} = \begin{bmatrix}
            T & 0 & 0 & \ddots\\
            Q^{\frac{1}{2}} & 0 & 0 & \ddots\\
            0 & S_{1} & 0 & \ddots\\
            0 & 0 & S_{2} & \ddots\\
            \ddots & \ddots & \ddots & \ddots 
        \end{bmatrix} \in \mathbf{B}\left( H\oplus \ell^{2}(\mathbb{N},\overline{\mathcal{R}(Q^{\frac{1}{2}})}) \right),
    \end{equation*}
    where $ Q\in \mathbf{B}(H) $ is also as in Lemma \ref{LemMConcaveAreQIsometries} and $ (S_{j})_{j\in \mathbb{N}_{1}}\subset \mathbf{B}(\overline{\mathcal{R}(Q^{\frac{1}{2}})}) $ is a certain uniformly bounded sequence of positive and invertible operators; moreover, by Corollary \ref{CorMinimalDilations}, this dilation is also minimal. Let $ T\in \mathbf{B}(H) $ be a 2-concave non-isometric operator. We will show that for such a $ T $, the dilations $ W $ and $ W^{\prime} $ are not isomorphic. Suppose to the contrary that there exists a unitary operator $ V\colon H\oplus \ell^{2}(\mathbb{N},\overline{\mathcal{R}(Q^{\frac{1}{2}})}) \to H\oplus \ell^{2}(\mathbb{N},\overline{\mathcal{R}(U)}) $ making them isomorphic. Then, by Lemma \ref{LemMatrixRepresentationOfPowers},
    \begin{align*}
        \lVert VW(h,0,0,\ldots)\rVert^{2} &= \lVert W(h,0,0,\ldots)\rVert^{2} \\
        &=\lVert (Th,Q^{\frac{1}{2}}h,0,\ldots)\rVert^{2} \\
        &= \lVert Th\rVert^{2} + \lVert Q^{\frac{1}{2}}h\rVert^{2}, \quad h\in H,
    \end{align*}
    and
    \begin{align*}
        \lVert W^{\prime} V(h,0,0,\ldots)\rVert^{2} &= \lVert W^{\prime}(h,0,0,\ldots)\rVert^{2}\\
        &= \lVert (Th,Uh,0,\ldots)\rVert^{2} \\
        &= \lVert Th\rVert^{2}+\lVert Uh\rVert^{2}, \quad h\in H.
    \end{align*}
    Since $ \lVert VW(h,0,0,\ldots)\rVert = \lVert W^{\prime} V(h,0,0,\ldots)\rVert $, we have $ \lVert Q^{\frac{1}{2}}h\rVert = \lVert Uh\rVert $ for every $ h\in H $. Combining this with the equality $ U = (Q-\beta_{1}(T))^{\frac{1}{2}} $, we obtain that $ \beta_{1}(T) = 0 $ or, in other words, $ T $ is isometric. This gives us a contradiction.
\end{example}
	\bibliographystyle{plain}
	\bibliography{references}
\end{document}